\documentclass[12pt,leqno]{amsart}
\usepackage{amsfonts,amsthm,amsmath,xcolor,enumitem,easyReview}
\usepackage{soul}

\theoremstyle{plain}
\newtheorem{thm}{Theorem}[section]
\newtheorem{prop}{Proposition}[section]
\newtheorem{lem}{Lemma}[section]

\theoremstyle{definition}
\newtheorem{df}{Definition}[section]
\newtheorem{rem}{Remark}[section]
\newtheorem{ex}{Example}[section]

\newcommand{\FF}{\mathbb{F}}

\newcommand{\RR}{\mathbb{R}}

\newcommand{\Z}{\mathbb{Z}}

\newcommand{\D}{\mathcal{D}}
\newcommand{\R}{\mathbb{R}}

\newcommand{\la}{\langle}
\newcommand{\ra}{\rangle}

\newcommand{\I}{\mathcal{I}}

\newcommand{\bbinom}[2]{\left[\!\begin{array}{c} #1 \\ #2 \end{array}\!\right]}

\def\bm#1{\mathbf{#1}}

\DeclareMathOperator{\supp}{supp}
\DeclareMathOperator{\Supp}{Supp}

\DeclareMathOperator{\wt}{wt}

\DeclareMathOperator{\Harm}{Harm}

\DeclareMathOperator{\Mat}{Mat}

\begin{document}

\title{{Harmonic higher and extended weight enumerators}}

\author[Britz]{Thomas Britz}
\address
{
	School of Mathematics and Statistics,
	University of New South Wales\\
	Sydney, NSW 2052, Australia
}
\email{britz@unsw.edu.au}

\author[Chakraborty]{Himadri Shekhar Chakraborty*}
\thanks{*Corresponding author}
\address
{
	Department of Mathematics, 
	Shahjalal University of Science and Technology\\ Sylhet-3114, Bangladesh\\
}
\email{himadri-mat@sust.edu}

\author[Miezaki]{Tsuyoshi Miezaki}
\address
{
	Faculty of Science and Engineering, 
	Waseda University, 
	Tokyo 169-8555, Japan\\
}
\email{miezaki@waseda.jp}

\date{}
\maketitle

\begin{abstract}
In this paper, we present the harmonic generalizations of 
well-known polynomials of codes over finite fields, 
namely the higher weight enumerators and the extended weight enumerators,
and we derive the correspondences between these weight enumerators.
Moreover, we present the harmonic generalization of Greene's Theorem 
for the higher (resp. extended) weight enumerators. 
As an application of this Greene's-type theorem,
we provide the MacWilliams-type identity for 
harmonic higher weight enumerators of codes over finite fields. 
Finally, we use this new identity to give a new proof of the Assmus-Mattson Theorem for
subcode supports of linear codes over finite fields
using harmonic higher weight enumerators.
\end{abstract}

{\small
\noindent
{\bfseries Key Words:}
Tutte polynomials, weight enumerators, matroids, codes, harmonic functions.\\ \vspace{-0.15in}

\noindent
2010 {\it Mathematics Subject Classification}. 
Primary 11T71;
Secondary 94B05, 11F11.\\ \quad
}

%\noindent
%2010 {\it Mathematics Subject Classification}.
%Primary 11F30;
%Secondary 20D08, 11F27.\\ \quad

\section{Introduction}

The classification of codes is one of the most challenging problems
in the theory of algebraic coding for the last few decades.
Mathematicians found a way to solve this problem when MacWilliams~\cite{MacWilliams} introduced 
a powerful identity for codes in 1963, now known as the MacWilliams identity.
This identity plays an important role in classifying codes.
The identity also allowed Assmus and Mattson~\cite{AsMa69} to prove 
a famous theorem now known as the Assmus-Mattson Theorem
that describes how linear codes can generate $t$-designs. 

In the decade prior to these results,
Tutte~\cite{Tutte1954} presented a two-variable polynomial of a graph, 
since known as the Tutte polynomial, 
that not only can reveal the internal structure of a graph
but also can be used to determine several other graph polynomials; see also~\cite{Tutte1967}.
Later, Crapo~\cite{Crapo} gave an appropriate generalization of
the Tutte polynomial by defining it for matroids.
Greene~\cite{Greene1976} presented a rare and unexpected 
relation between the weight enumerator of a 
linear code over finite field and the Tutte polynomial of a matroid.  
Moreover, using this relation, 
Greene~\cite{Greene1976} also gave a combinatorial proof of 
the MacWilliams identity for the weight enumerator of a code. 
Wei~\cite{Wei1991} introduced the notion of subcode weights in the study of coding theory, 
and Kl{\o}ve~\cite{Klove1992} gave a remarkable generalization of the MacWilliams identity
for these subcode weights.
Many authors have studied Greene's Theorem 
between the linear codes and matroids; 
for instance, see~\cite{CMO20xx}.
Of note, Britz~\cite{Britz2007} generalized Greene's Theorem
with respect to the rank generating function of a matroid  
and the higher support weight enumerator of a linear code over a finite field.
By applying this research, 
Britz and Shiromoto~\cite{BrSh2008} generalized the Assmus-Mattson Theorem
with respect to subcode supports of linear codes over finite fields; 
see also \cite{BrBrShSo07,BrRoSh2009,ByCeIoJu2024,Tanabe2001}.

Delsarte~\cite{Delsarte} introduced discrete harmonic functions on a finite set.
Later, Bachoc~\cite{Bachoc} obtained 
a striking generalization of the MacWilliams identity 
by introducing the concept of the harmonic weight enumerators of a binary code. 
The harmonic weight enumerator of code is a weight enumerator
of a code associated to a discrete harmonic function.
Recently, Chakraborty, Miezaki and Oura~\cite{CMO20xx}
introduced harmonic Tutte polynomials of matroids,
and presented the harmonic generalization of Greene's Theorem.
Subsequently, a demi-matroid analogue of harmonic Tutte polynomials
was discussed in~\cite{BrChIsMiTa2024}.

In this paper, we introduce the notion of harmonic higher 
(resp.\ extended) weight enumerators of codes over finite fields.
We present fundamental theorems regarding these weight enumerators, 
including their Greene's-type theorem.
As an application of this theorem,
we provide the MacWilliams-type identity for 
harmonic higher weight enumerators over finite fields. 
Finally, we use this new identity for harmonic higher weight enumerators 
to give a new proof of the Assmus-Mattson Theorem 
for subcode supports of linear codes over finite fields.

This paper is organized as follows. 
%\textcolor{red}{In preparation...}
In Section~\ref{Sec:Preli}, 
we present basic definitions and properties 
that are frequently used in this paper, 
of discrete harmonic functions, 
linear codes over finite fields and 
matroids.
In Section~\ref{Sec:MacWilliams}, 
we define the harmonic higher (resp.\ extended) weight enumerators 
for codes over finite fields,
and give formulae to compute these; 
see Theorems~\ref{Thm:reinter} and~\ref{Thm:ExtendedReinter}.
We also provide the relationships between these weight enumerators;
see Theorems~\ref{Thm:ExtendedtoHigher} and~\ref{Thm:HigherToExtended}. 
Moreover, we present the MacWilliams-type identity for 
harmonic higher (resp.\ extended) weight enumerators of codes 
over finite fields; 
see Theorems~\ref{Thm:HarmHigherMac} and~\ref{Thm:HarmTupleMacWilliams.}. 
In Section~\ref{Sec:Greene},
we give the Greene-type identity for the 
harmonic higher (resp.\ extended) weight enumerators for codes over finite fields;
see Theorems~\ref{Thm:HarmExGreene} and~\ref{Thm:HarmHigherGreene}.
We apply the Greene-type identity for harmonic extended weight enumerators
of codes over finite fields 
to prove the MacWilliams-type identity for the
harmonic higher weight enumerators of codes over finite fields. 
Finally, in Section~\ref{Sec:Design}, 
we use harmonic higher weight enumerators to 
give a new proof of the Assmus-Mattson Theorem for
subcode supports of linear codes over finite fields; 
see Theorem~\ref{Thm:AssmusMattson} and its proof.

\section{Preliminaries}\label{Sec:Preli}

In this section, we present basic definitions and notation for linear codes 
and matroids that are frequently needed in this paper. 
We mostly follow the {definitions and notation} of~\cite{CR1970, MS1977, Oxley1992}.
We also present definitions and properties of {(discrete) harmonic functions}, 
following~\cite{Bachoc,BachocNonBinary,Delsarte} .

\subsection{Discrete harmonic functions}

Let $E := \{1,2,\ldots,n\}$ be a finite set.
%(which represents the set of edges in this note).
We define  
$E_{d} := \{ X \subseteq E : |X| = d\}$
for $d = 0,1, \ldots, n$. 
The set of all subsets of $E$ is denoted by $2^{E}$.
We denote by 
$\R 2^{E}$ and $\R E_{d}$
the real vector spaces spanned by the elements of  
$2^{E}$ and $E_{d}$,
respectively. 
An element of 
$\R E_{d}$
is denoted by
\begin{equation}\label{Equ:FunREd}
	f :=
	\sum_{Z \in E_{d}}
	f(Z) Z
\end{equation}
and is identified with the real-valued function on 
$E_{d}$
given by 
$Z \mapsto f(Z)$. 
Such an element 
$f \in \R E_{d}$
can be extended to an element 
$\widetilde{f}\in \R 2^{E}$
by setting, for all 
$X \in 2^{E}$,
\begin{equation}\label{Equ:TildeF}
	\widetilde{f}(X)
	:=
	\sum_{Z\in E_{d}, Z\subseteq X}
	f(Z).
\end{equation}
{Note that $\widetilde{f}(\emptyset) = f(\emptyset)$ when $d=0$,
and that $\widetilde{f}(\emptyset) = 0$ otherwise}. 
If an element 
$g \in \R 2^{E}$
is equal to $\widetilde{f}$  
for some $f \in \R E_{d}$, 
then we say that $g$ has degree~$d$. 
{The differentiation operator $\gamma$ on $\mathbb{R}E_d$ 
is defined by} linearity from the identity
\begin{equation}\label{Equ:Gamma}
	\gamma(Z) := 
	\sum_{Y\in {E}_{d-1}, Y\subseteq Z} 
	Y
\end{equation}
for all 
$Z \in E_{d}$
and for all $d=0,1, \ldots,n$.
Also, $\Harm_{d}(n)$ is the kernel of~$\gamma$:
\begin{equation}\label{Equ:Harm}
	\Harm_{d}(n) 
	:= 
	\ker
	\left(
	\gamma\big|_{\R E_{d}}
	\right).
\end{equation}

 \begin{rem}[\cite{Bachoc,Delsarte}]\label{Rem:Gamma}
	Let $f \in \Harm_{d}(n)$ and $i \in \{0,1,\ldots,d-1\}$.
	Then 
	$\gamma^{d-i}(f) = 0$. 
    This means from definition~(\ref{Equ:Gamma})
	that
	$$\sum_{X \in E_{i}}\left(\sum_{\substack{Z \in E_{d}, X \subseteq Z}} f(Z)\right) X = 0.$$
	This implies that $\sum_{\substack{Z \in E_{d}, X \subseteq Z}} f(Z) = 0$
	for any $X \in E_{i}$.
%   {\color{blue}Is there an easy and brief explanation to show this equation?}
\end{rem}

\begin{rem}\label{Rem:New}
	Let~$f \in \Harm_{d}(n)$. 
	Since $\sum_{Z \in E_{d}} f(Z) = 0$, 
	it is easy to check from~(\ref{Equ:Gamma}) that
	%$\sum_{X \subseteq E, |X| = t} \widetilde{f}(X) = 0$
	$\sum_{X \in E_{t}} \widetilde{f}(X) = 0$, where
	$1 \leq d \leq t \leq n$.
\end{rem}

\begin{ex}
	Let $E = \{1,2,3,4\}$ and $d =2$. 
    Let $f \in \RR E_{2}$ be the element
	\[
		f 
		= 
		a_1 \{1,2\}
		+
		a_2 \{1,3\}
		+
		a_3 \{1,4\}
		+
		a_4 \{2,3\}
		+
		a_5 \{2,4\}
		+
		a_6 \{3,4\}.
	\]
	{If $X = \{1,3,4\}$,
	then $\widetilde{f}(X) = a_2+a_3+a_6$.
    Now applying differential operator $\gamma$ on $f$, we have}
    \[
        \gamma(f)
        =
		(a_1+a_2+a_3) \{1\}
		+
		(a_1+a_4+a_5)\{2\}
		+
		(a_2+a_4+a_6)\{3\}
		+
		(a_3+a_5+a_6)\{4\}.    
    \]
	\noindent
%	Suppose that $f \in \Harm_{2}(4)$; 
%    then $\gamma(f) = 0$, 
%    so 
%    \[
%      a_1+a_2+a_3 = a_1+a_4+a_5 = a_2+a_4+a_6 = a_3+a_5+a_6 = 0.
%	\]
%	Solving the above equations, 
    {It follows that $f\in\Harm_{2}(4)$
    if and only if 
    $a_{3} = a_{4} = -a_{1}-a_{2}$, $a_{5} = a_{2}$ and $a_{6} = a_{1}$.}
%	\begin{multline*}
%		\Harm_{2}(4)
%		\ni
%		f
%		=
%		a_1 \{1,2\}
%		+
%		a_2 \{1,3\}
%		+
%		(-a_1-a_2) \{1,4\}\\
%		+
%		(-a_1-a_2) \{2,3\}
%		+
%		a_2 \{2,4\}
%		+
%		a_1 \{3,4\}.		
%	\end{multline*}
\end{ex}

\subsection{Linear codes}

Let $\FF_{q}$ be a finite field of order~$q$, where~$q$ is a 
prime power.  
Then $\FF_{q}^{n}$ denotes the vector space of 
dimension~$n$ with the usual inner product:
$\bm{u}\cdot\bm{v} := u_{1}v_{1} + \cdots + u_{n}v_{n}$
for $\bm{u},\bm{v} \in \FF_{q}^{n}$,
where
$\bm{u} = (u_{1},\ldots,u_{n})$ and $\bm{v} = (v_{1},\ldots,v_{n})$.
For $\bm{u} \in \FF_{q}^{n}$, we call
$\supp(\bm{u}) := \{i \in E \mid u_{i} \neq 0\}$
the \emph{support} of $\bm{u}$,
and
$\wt(\bm{u}) := |\supp(\bm{u})|$
the \emph{weight} of $\bm{u}$. 
Similarly, the \emph{support} and \emph{weight}
of each subset of vectors $D \subseteq \FF_{q}^{n}$ are defined as follows:
\begin{align*}
	\Supp(D) 
	& :=
	\bigcup_{\bm{u} \in D}
	\supp(\bm{u})\\
	\wt(D) 
	& :=
	|\Supp(D)|.
\end{align*}

An $\FF_{q}$-\emph{linear code} of length~$n$ is a linear subspace of $\FF_{q}^{n}$. 
Moreover, an $[n,k,d_{\min}]$~code 
is an $\FF_q$-linear code of length~$n$ with dimension~$k$ 
and minimum weight~$d_{\min}$.
Frequently, we call $[n,k,d_{min}]$~codes $[n,k]$~codes.
The \emph{dual code}~$C^{\perp}$ of an $\FF_{q}$-linear code~$C$  
can be defined as follows:
\[
	C^{\perp}
	:=
	\{
	\bm{u} \in \FF_{q}^{n}
	\mid
	\bm{u} \cdot \bm{v} = 0
	\mbox{ for all }
	\bm{v} \in C
	\}.
\] 

Let $C$ be an $[n,k]$ code. 
Let $r,i$ be positive integers such that $r \leq k$ and $i \leq n$.
Now we define
\begin{align*}
	\D_{r}(C)
	& :=
	\{
	D \mid D \text{ is an } [n,r] \text{ subcode of } C 
	\},\\
	\mathcal{S}_{r}(C)
	& :=
	\{
	\Supp(D) \mid D \in \D_{r}(C)
	\},\\
	\mathcal{S}_{r,i}(C)
	& :=
	\{
	X \in \mathcal{S}_{r}(C) \mid |X| = i
	\}.
\end{align*}
Note that, in general, 
the sets $\mathcal{S}_{r}(C)$ and $\mathcal{S}_{r,i}(C)$ are multisets.
Also, the support of a nonzero codeword $\bm{v}$ 
is identical to that of its span~$\la \bm{v} \ra$. 
Conversely, each member of $\mathcal{S}_{1}(C)$ is the
support of some nonzero codeword. 
Therefore, the sets $\mathcal{S}_{1}(C),\ldots,\mathcal{S}_{k}(C)$ 
extend the notion of codeword support.

Let $C$ be an $\FF_{q}$-linear code of length~$n$.
Then the $r$-th generalized Hamming weight of~$C$
for any~$r$, $1 \leq r \leq k$, is defined as
\[
	d_{r}
	:=
	d_{r}(C)
	:=
	\min
	\{
		\wt(D)
		\mid
		D \in \mathcal{D}_{r}(C)
	\}.
\]
Moreover, the \emph{$r$-th higher weight distribution} of $C$ is the sequence 
$$\{A_{i}^{(r)}(C)\mid i=0,1, \dots, n \},$$ 
where $A_{i}^{(r)}(C)$ is the number of subcodes with a given weight~$i$
and dimension~$r$. That is,
\[
	A_{i}^{(r)}(C)
	:=
	\#
	\{
		D\in \mathcal{D}_{r}(C)
		\mid
		\wt(D) = i
	\}.
\]
Then the polynomial
\[
	W_{C}^{(r)}(x,y) 
	:= 
	\sum^{n}_{i=0} 
	A_{i}^{(r)}(C) x^{n-i} y^{i}
\]
is called the \emph{$r$-th higher weight enumerator} of $C$
and satisfies the following MacWilliams-type identity
stated in~\cite[{Theorem 5.14}]{JP2013}:
\[
	W_{C^{\perp}}^{(r)}(x,y)
	=
	\sum_{j=0}^{r}
	\sum_{\ell=0}^{j}
	(-1)^{r-j}
	\frac{q^{{r-j \choose 2}-j(r-j)-\ell(j-\ell)-jk}}{[r-j]_{q} [j-\ell]_{q}}
	W_{C}^{(\ell)}(x+(q^{j}-1)y,x-y).
\]
\noindent
Here, for all integers $a,b\geq 0$, 
\[
	[a]_q 
	:= 
	[a,a]_q
	\qquad\textrm{where}\qquad
	[a,b]_{q} 
	:=
	\prod_{i=0}^{b-1}
	\big(q^{a}-q^{i}\big).
\]

Let $C$ be an $\FF_{q}$-linear code of length~$n$.
Then for an arbitrary subset $X \subseteq E$ and integer~$r \geq 0$,
we define
\begin{align*}
	C(X) 
	& := 
	\{
	\bm{u} \in C 
	\mid 
	u_{i} = 0 \text{ for all } i \in X
	\}\,,\\
	\ell(X) 
	& := 
	\dim C(X)\,,\\
	B_{X}^{(r)}(C)
	& := 
	%	q^{\ell(J)}-1
	\#
	\{
	D \subseteq C(X)
	\mid 
	D \text{ is a subspace of dimension } r
	\}\,. 
\end{align*}
Also, define 
\begin{align*}
%	[a,b]_{q} 
%	&:=
%	\prod_{i=0}^{b-1}
%	q^{a}-q^{i},\\
%	[a]_{q}
%	&:=
%	[a,a]_{q},\\
	{\bbinom{a}{b}}_{q}
	&:=
	\frac{[a,b]_{q}}{[b]_{q}}
      =
    \prod_{i=0}^{b-1}\frac{ q^{a}-q^{i}}{q^{b}-q^{i}}\,.	
\end{align*}

The following two {lemmas} provide important and useful identities.
{The second lemma follows immediately from~\cite[Theorem 25.2]{vLiWi01}.}

\begin{lem}\label{Rem:BXrC}
	${B_{X}^{(r)}(C) = \bbinom{\ell(X)}{r}}_{q}$.
\end{lem}

\begin{lem}\label{Rem:TBritz}
	
	$[a,b]_{q} 
	= 
	\displaystyle
    \sum_{i=0}^{b} 
	{\bbinom{b}{i}}_{q}
	(-1)^{b-i}
	q^{\binom{b-i}{2}}
	(q^a)^i$.
\end{lem}

%% [TB:] This lemma is not used in the paper.
%%
%\begin{lem}[\cite{JP2013}]\label{LemRank}
%	Let $C$ be an $[n,k]$ linear code with generator matrix~$G$. 
%	Assume that the columns of $G$ are indexed by the set $E$.
%	Let $G_{X}$ be the $k \times t$ submatrix of $G$ consisting of the columns of $G$ indexed by $X \in E_{t}$, 
%	and let $\rho(X)$ be the rank of $G_{X}$. 
%	Then $\ell(X) = k - \rho(X)$.
%	%	$\rho(J) = k - \ell(J)$.
%\end{lem}

%A code $C$ is called \emph{self-dual} if $C = C^\perp$. 
%It is well known that 
%the length~$n$ of a self-dual code is even and the dimension is $n/2$.
%To study self-dual codes in detail, we refer the readers 
%to~\cite{BMS1972, CS1999, Gleason, MMS1972, NRS}. 
%A self-dual code~$C$ over~$\FF_2$ or~$\FF_4$ of length $n\equiv 0 \pmod 2$ having even weight is called \emph{Type}~$\I$ and \emph{Type}~$\IV$, respectively. A self-dual code $C$ over~$\FF_2$ of length $n\equiv 0\pmod 8$ is called \emph{Type}~$\II$ if the weight of each codeword of~$C$ is multiple of~$4$. Finally, a self-dual code $C$ over~$\FF_3$ of length $n\equiv 0\pmod 4$ is called \emph{Type}~$\III$ if the weight of each codeword of~$C$ is multiple of~$3$.

\subsection{Matroids}

Matroids can be defined in several equivalent ways.
We prefer the definition which is in terms of independent sets. 
%Let $E$ be a finite set of cardinality~$n$ and 
%$2^{E}$ denotes the set of all subsets of $E$. 
A (finite) \emph{matroid} 
$M$ is an ordered pair $ (E, \I) $ consisting of set $E$
and a collection $\I$ of subsets of $E$ 
satisfying the following conditions:  
\begin{itemize}
	\item[(I1)] 
	$ \emptyset \in \I $,
	\item[(I2)] 
	if $ I \in \I $ 
	and 
	$ J \subset I $, 
	then 
	$ J \in \I $, and
	\item[(I3)] 
	if $ I, J \in \I $ 
	with $ |I| < |J| $, 
	then there exists 
	$ j \in J \setminus I $ 
	such that 
	$ I \cup \{ j \} \in \I $.
\end{itemize}

The elements of $\I$ are called the \emph{independent} sets of $M$, 
and $E$ is called the \emph{ground set} of $M$. 
A subset of the ground set $E$ 
that does not belong to $\I$ is called \emph{dependent}. 
An independent set is called a \emph{basis} 
if it is not contained in any other independent set. 
The \emph{dual matroid} of~$M$ is denoted by~$M^{\ast}$
and has as its independent sets all subsets of the complements of each basis of $M$.

It follows from axiom (I3) that the cardinalities of all bases in a
matroid $M$ are equal; this cardinality is called the \emph{rank} of $M$. 
%These maximal independent sets are called the \emph{bases} of $M$. 
The \emph{rank} $\rho(X)$ of an arbitrary subset $X$ of $E$ 
is the size of the largest independent subset of~$X$.
That is, 
$
\rho(X) 
:= 
\max 
\{ 
|I| : 
I \in \I
\text{ and }
I \subseteq X
\}.
$
%This implies 
%$ \rho $ 
%maps 
%$ 2^{E} $ 
%into 
%$\mathbb{Z}$. 
%This function $ \rho $ is called the \emph{rank function} of $M$. 
In particular, 
$\rho(\emptyset) = 0$.
We call
$\rho(E)$ the rank of~$M$.
We refer the readers to~\cite{Oxley1992} for more information on matroids. 

\begin{df}
	Let~$M$ be a matroid on the set $E$ having a rank function $\rho$.
	The \emph{Tutte polynomial} of $M$ is defined as follows:
	%in~\cite{Tutte1947,Tutte1954,Tutte1967}:
	\[
	T(M;x,y)
	:=
	{\sum_{X \subseteq E}}
	(x-1)^{\rho(E)-\rho(X)}
	(y-1)^{|X|-\rho(X)}.
	\]
\end{df}

Let $A$ be a $k \times n$ matrix over a finite field $\FF_{q}$. 
This gives a matroid $M[A]$ on the set $E$
in which a set $I$ is independent if and only if the family of 
columns of~$A$ whose indices belong to $I$ is linearly independent. 
Such a matroid is called a \emph{vector matroid}.
Let $G$ be a generator matrix of an $\FF_{q}$-linear code $C$ of length~$n$. 
Then the vector matroid~$M_{C} := M[G]$ is a matroid on~$E$.
Note that $M_{C}$ independent of the choice of the generator matrix~$G$. 
Moreover, it is well known that the dual matroid corresponds to dual code: 
$(M_{C})^\ast = M_{C^\perp}$.

\section{Coding theory associated to harmonic functions}\label{Sec:MacWilliams}

{Bachoc~\cite{Bachoc} first defined the notion of 
harmonic weight enumerators for binary codes 
and presented a MacWilliams-type identity for these enumerators. 
Later, several articles extended this concept 
for linear codes over finite fields~\cite{BachocNonBinary,Tanabe2001} 
as well as over finite rings~\cite{BrChIsMiTa2024}.}
In this section, 
we introduce the harmonic generalization of higher (resp. extended) 
weight enumerators for codes over finite fields.

{\begin{df}\label{DefHarmWeight}
	Let $C$ be an $\FF_{q}$-linear code of length~$n$. 
	Let $f\in\Harm_{d}(n)$, where $d \neq 0$. 
	Then the \emph{harmonic weight enumerator} of $C$ associated to~$f$ is
	defined as follows:	
	\[	
		W_{C,f}(x,y) 
		:=
		\sum_{i=0}^{n}
		A_{i,f}(C)
		x^{n-i}
		y^{i},
	\]
	where 
	\[
		A_{i,f}(C)
		:= 
		\sum_{\bm{u} \in C, \wt(\bm{u}) = i} 
		\widetilde{f}(\supp(\bm{u})).
	\]
\end{df}}

\subsection{Harmonic higher weight enumerators}

$ $\\[1mm]
Let $C$ be an $[n,k]$ code over $\FF_{q}$ and let $f\in\Harm_{d}(n)$, where $d\neq 0$. 

\begin{df}\label{DefHarmWeightBachoc}
	The \emph{harmonic $r$-th higher weight enumerator} of $C$ associated to $f$ is
	defined as follows:	
	\[
		W_{C,f}^{(r)}(x,y) 
		:=
%		\sum_{{\bf u}\in C}
%		\acute{f}({\bf u})
%		x^{n-\wt({\bf u})}
%		y^{\wt({\bf u})}
%		=
		\sum^{n}_{i=0} 
		A_{i,f}^{(r)}(C) x^{n-i} y^{i},
	\]
	where 
	\[
		A_{i,f}^{(r)}(C)
		:= 
		\sum_{\substack{D \in \mathcal{D}_{r}(C),\\\wt(D) = i}} 
		\widetilde{f}(\supp(D))\,.
	\]
\end{df}

\begin{rem}
	Clearly,
	$A_{0,f}^{(r)}(C) = 0$ for all $0 \leq r \leq k$.
\end{rem}

\begin{rem}
%	Let $C$ be an $[n,k]$ code and let $f \in \Harm_{d}(n)$, where $d \neq 0$.
	Since every $1$-dimensional subspace of $C$ contains $q-1$ nonzero codewords, 
    $(q - 1)A_{i,f}^{(1)}(C) = A_{i,f}(C)$ for $0 \leq i \leq n$. 
	Therefore,
	\[
		W_{C,f}(x,y) =W_{C,f}^{(1)}(x,y)\,.
	\]
\end{rem}

\begin{rem}
	If $\deg f = 0$, then
	{$W_{C,f}^{(r)}(x,y) = f(\emptyset)\, W_{C}^{(r)}(x,y)$}.
%    then the harmonic $r$-th higher weight enumerator $W_{C,f}^{(r)}(x,y)$ is 
%    the $r$-th higher weight enumerator $W_{C}^{(r)}(x,y)$.
\end{rem}

\begin{lem}[\rm\cite{Bachoc}]\label{Lem:Bachoc}
	{Let $f \in \Harm_{d}(n)$ and $J \subseteq E$, 
	and define}
	\[
		{f^{(i)}(J)
		:=
		\sum_{\substack{Z \in E_{d},\\ |J \cap Z| = i}}
		f(Z).}
	\]
	{Then for all 
	$0 \leq i \leq d$,
	$f^{(i)}(J) = (-1)^{d-i} \binom{d}{i} \widetilde{f}(J)$.}
\end{lem}

\begin{rem}\label{Rem:BachocLem}
	{From the definition of $\widetilde{f}$ 
	for $f \in \Harm_{d}(n)$,
	we have $\widetilde{f}(J) = 0$
	for each $J \in 2^{E}$ such that $|J| < d$. 
	Let $J\in 2^{E}$ and set $I = E \setminus J$.
	Then by Lemma~\ref{Lem:Bachoc},}
	\begin{align*}
		{\widetilde{f}(J)
		=
		\sum_{\substack{Z \in E_{d},\\Z \subset J}}
		f(Z)
		=
		\sum_{\substack{Z \in E_{d},\\ |I \cap Z|=0}}
		f(Z)
		=
		f^{(0)}(I)
		=
		(-1)^{d} \widetilde{f}(E \setminus J).}		
	\end{align*}	
	{Moreover, 
	we have from the above equality that if $|J| > n-d$, 
	then $\widetilde{f}(J) = 0$.}
\end{rem}

{Let $C$ be an $[n,k]$ code over~$\FF_{q}$. 
Then it is immediate from Remark~\ref{Rem:BachocLem} that
the harmonic $r$-th higher weight enumerator
$W_{C,f}^{(r)}(x,y)$ of~$C$ 
associated to $f\in \Harm_{d}(n)$ for any~$r$,
$1 \leq r \leq k$ can be written as:}
%From the above discussion and the first part of Theorem~\ref{Thm:HarmHigherMac}, 
%we see that
\[
	{W_{C,f}^{(r)}(x,y) 
	= 
	(xy)^{d} Z_{C,f}^{(r)}(x,y),}
\]
{where $Z_{C,f}^{(r)}$ is a homogeneous polynomial of degree $n-2d$
of the form as follows:}

\[
	{Z_{C,f}^{(r)}(x,y) 
	= 
	\sum_{i=d}^{n-d}
	A_{i,f}^{(r)}(C)
	x^{n-i-d}y^{i-d}.}
\]

\begin{ex}\label{Ex:HarmWeight}
	Let $C$ be the binary linear code of length~$3$ with codewords
	\[
		(0,0,0), (0,0,1), (1,1,0), (1,1,1)\,.
	\]
	{We consider $f \in \Harm_{1}(3)$, where
	$f = a\{1\}+b\{2\}-(a+b)\{3\}$.}
	Then by direct computation, 
	we find the harmonic weight enumerator of~$C$ associated to~$f$ 
	to be
	\[
		W_{C,f}^{(1)} 
		= 
		-(a+b) x^2 y + (a+b) x y^2 = (xy)^{1} Z_{C,f}^{(1)},
	\]
	where $Z_{C,f}^{(1)} = (a+b)(y-x)$.
\end{ex}

\begin{thm}[MacWilliams type identity]\label{Thm:HarmHigherMac} 
	{Let $C$ be an $[n,k]$ code over $\FF_{q}$,
	and let $f \in \Harm_{d}(n)$.
	Then }
	\begin{align*}
		Z_{C^{\perp},f}^{(r)}
		(x,y)
		=  
		\sum_{j=0}^{r}
		\sum_{\ell=0}^{j}
		(-1)^{r+d-j}
		&\frac{q^{{r-j \choose 2}-j(r-j)-\ell(j-\ell)-j(k-d)}}{[r-j]_{q} [j-\ell]_{q}}\\
		&Z_{C,f}^{(\ell)}
		\left( 
			{x+(q^{j}-1)y}, 
			{x-y}
		\right).
	\end{align*}
\end{thm}

%In order to give a proof of the above theorem, %\ref{Thm:HarmTutteMacIden.} 
{To pave the way for the proof of Theorem~\ref{Thm:HarmHigherMac}
that we give 
%in Section~\ref{Sec:Greene},
in Section~\ref{Sec:Greene},
we begin by presenting some propositions and theorems
%harmonic functions
%and their associated polynomials 
related to~$Z_{C,f}^{(r)}(x,y)$.}

%Now we have the following proposition.

\begin{prop}\label{Prop:Connection}
	Let $f \in \Harm_{d}(n)$ and $X \subseteq E$.
	Define
	\[
		B_{t,f}^{(r)}(C)
		:= 
		\sum_{X \in E_{t}} 
		\widetilde{f}(X)B_{X}^{(r)}(C).
	\]
	{Then
	we have the following relation between 
	$B_{t,f}^{(r)}(C)$ and $A_{i,f}^{(r)}(C)$:}
	\[
		B_{t,f}^{(r)}(C) 
		= 
		(-1)^{d}
		\sum_{i=d}^{n-t} 
		{\binom{n-d-i}{t-d}} A_{i,f}^{(r)}(C)
	\]
	if $d \leq t \leq n-d$;
	otherwise, $B_{t,f}^{(r)}(C) = 0$.
\end{prop}

\begin{proof}
	It is immediate from (\ref{Equ:TildeF}) 
	%the definition of $\widetilde{f}$ 
	and Remark~\ref{Rem:BachocLem} that
	$B_{t,f}^{(r)}(C) = 0$ for $0 \leq t < d$ and~$n-d < t \leq n$.
	We now focus on $t$ with 
	%only need to show for 
	$d \leq t \leq n-d$.
	%	\[
	%		B_{t,f}
	%		=
	%		(-1)^{d}
	%		\sum_{i=d}^{n-t} 
	%		\binom{n-d-i}{t-d} A_{i,f}.
	%	\]
	%	From the construction of $B_{t,f}$ and 
	{By  Remark~\ref{Rem:BachocLem}, we have}
%    and Remark~\ref{Rem:BXrC},
	%	it is clear that
	\begin{align*}
		B_{t,f}^{(r)}(C) 
		& = 
		\sum_{X \in E_{t}} 
		\widetilde{f}(X)B_{X}^{(r)}(C)\\
		& =
		(-1)^{d}
		\sum_{X\in E_{t}}
		\widetilde{f}(E\setminus X) B_{X}^{(r)}(C)\\
		& =
		(-1)^{d}
		\sum_{X\in E_{t}}\,
		\sum_{\substack{D\in \mathcal{D}_r(C),\\\Supp(D)\subseteq E\setminus X}} \widetilde{f}(E\setminus X)\\
		& =
		(-1)^{d}
		\sum_{D\in \mathcal{D}_r(C)}\,\,
        \sum_{\substack{X\in E_{t},\\X\subseteq E\setminus\Supp(D)}}
		\widetilde{f}(E\setminus X)\\
		& =
		(-1)^{d}
        \sum_{i=d}^{n-t}\,
		\sum_{\substack{D\in \mathcal{D}_r(C),\\\wt(D)=i}} \,\,
        \sum_{\substack{X\in E_{t},\\X\subseteq E\setminus\Supp(D)}}
		\widetilde{f}(E\setminus X)\\
		& =
		(-1)^{d}
        \sum_{i=d}^{n-t}\,
		\sum_{\substack{D\in \mathcal{D}_r(C),\\\wt(D)=i}} \,\,
        \sum_{\substack{Y\in E_{n-i-t},\\Y\subseteq E\setminus\Supp(D)}}
		\widetilde{f}(\Supp(D)\cup Y)\\
		& =
		(-1)^{d}
        \sum_{i=d}^{n-t}\,
		\sum_{\substack{D\in \mathcal{D}_r(C),\\\wt(D)=i}}\,\,
        \sum_{\substack{Y\in E_{n-i-t},\\Y\subseteq E\setminus\Supp(D)}}\,\,
      	\sum_{\substack{Z\in E_{d},\\Z\subseteq \Supp(D)\cup Y}} 
      	f(Z)\\
      	& =
      	(-1)^{d}
      	\sum_{i=d}^{n-t}\,
      	\sum_{\substack{D\in \mathcal{D}_r(C),\\\wt(D)=i}} \,\,
      	\sum_{j=0}^d \,\,
      	\sum_{\substack{Z\in E_{d},\\|Z\cap\Supp(D)|=j}}\,\,
      	\sum_{\substack{W\in E_{n-i-t-d+j},\\W\subseteq E\setminus(\Supp(D)\cup Z)}}
      	f(Z)\\
		& =
		(-1)^{d}
        \sum_{i=d}^{n-t}\,
		\sum_{\substack{D\in \mathcal{D}_r(C),\\\wt(D)=i}} \,\,
        \sum_{j=0}^d \binom{n-i-d+j}{n-i-t-d+j}
      	\sum_{\substack{Z\in E_{d},\\|\Supp(D)\cap Z|=j}}
        f(Z)\,.
	\end{align*}
    
    {\noindent Using inclusion-exclusion to count the 
    $t$-element subsets of an $x$-element set 
    that include some fixed $d$ elements gives the identity
    \begin{equation}\label{equ:inclexcl}
      \sum_{j=0}^d (-1)^j \binom{x-j}{t}\binom{d}{j} = \binom{x-d}{t-d}\,.
    \end{equation}
    By Lemma~\ref{Lem:Bachoc} and (\ref{equ:inclexcl}),}  
    it follows that
	\begin{align*}
		B_{t,f}^{(r)}(C) 
	    & = 
		(-1)^{d}
        \sum_{i=d}^{n-t}\,
		\sum_{\substack{D\in \mathcal{D}_r(C),\\\wt(D)=i}}\,
        \sum_{j=0}^d \binom{n-i-d+j}{t}
      	(-1)^{d-j}\binom{d}{j}\widetilde{f}(\Supp(D))\\
		& = 
		(-1)^{d}
        \sum_{i=d}^{n-t}\,
		\sum_{\substack{D\in \mathcal{D}_r(C),\\\wt(D)=i}}\,
        \sum_{j=0}^d (-1)^{j}
        \binom{n-i-j}{t} \binom{d}{j}
        \widetilde{f}(\Supp(D))\\
        & = 
		(-1)^{d}
		\sum_{i=d}^{n-t}\,
		\sum_{\substack{D \in \mathcal{D}_{r}(C),\\\wt(D) = i}} 
        \binom{n-d-i}{t-d}
		\widetilde{f}(\Supp(D))\\
        & = 
		(-1)^{d}
		\sum_{i=d}^{n-t} 
        \binom{n-d-i}{t-d}
        A_{i,f}^{(r)}(C)\,.\qedhere
    \end{align*}
\end{proof}

\noindent
Now we have the following result.
\begin{thm}\label{Thm:reinter}
	Let $C$ be an $[n,k]$ code $C$ over $\FF_{q}$,
    and let $f \in \Harm_{d}(n)$.
	Then
	\[
		Z_{C,f}^{(r)}(x,y)
		= 
		(-1)^{d}
		\sum_{t=d}^{n-d} 
		B_{t,f}^{(r)}(C) 
		(x-y)^{t-d} y^{n-t-d}\,.
	\]
\end{thm}

\begin{proof}
	By Proposition~\ref{Prop:Connection}
    and the binomial expansion of 
	$x^{n-i} = ((x-y)+y)^{n-i}$,
	we have
	\begin{align*}
		(-1)^{d}
		&
		\sum_{t=d}^{n-d} 
		B_{t,f}^{(r)}(C)
		(x-y)^{t-d} y^{n-t-d}\\ 
		= & 
		\sum_{t=d}^{n-d}\,
		\sum_{i = d}^{n-t} 
		{\binom{n-d-i}{t-d}}
		A_{i,f}^{(r)}(C)
		(x-y)^{t-d} y^{n-t-d}\\
		= & 
		\sum_{i=d}^{n-d} 
		A_{i,f}^{(r)}(C)
		\left( 
		\sum_{t=d}^{n-i} 
		{\binom{n-d-i}{t-d}}
		(x-y)^{t-d} y^{(n-d-i)-(t-d)}
		\right) 
		y^{i-d}\\
		= & 
		\sum_{i = d}^{n-d} 
		A_{i,f}^{(r)}(C)
		x^{n-d-i} y^{i-d}\\
		= &  
		\sum_{i = 0}^{n} 
		A_{i,f}^{(r)}(C)
		x^{n-i-d} y^{i-d}\\
		= & \,
		Z_{C,f}^{(r)}(x,y),
	\end{align*}
	since $A_{i,f}^{(r)}(C) = 0$ for $i<d$ and $i > n-d$.
\end{proof}

\subsection{Harmonic extended weight enumerators}\label{Sec:HarmExWtEnum}

Let~$m$ be a positive integer. 
Then any~$[n,k]$ code~$C$ over $\FF_{q}$
can be extended to a linear code $C{[q^m]}:=C\otimes \FF_{q^{m}}$
over~$\FF_{q^{m}}$ by considering all the $\FF_{q^{m}}$-linear combinations
of the codewords in~$C$. 
We call~$C{[q^m]}$ the $\FF_{q^{m}}$-\emph{extension code} of~$C$. 
%over~$\FF_{q^{m}}$.
Clearly, $\dim_{\FF_{q}} C(X) = \dim_{\FF_{q^{m}}} C[q^m](X)$ 
for all $X \subseteq E$ (see~\cite[{p.~34}]{JP2013}).
%Now we have the following result.

\begin{df}
	Let $C$ be an $[n,k]$ code $C$ over $\FF_{q}$,
    and let $f \in \Harm_{d}(n)$.
	Then the \emph{harmonic $q^m$-extended weight enumerator}
	of~$C$ associated to~$f$ can be defined as follows:
	\[
		W_{C,f}(x,y;q^m)
		:=
		\sum_{i=0}^{n}
		A_{i,f}^{q^m}(C)
		x^{n-i} y^{i},
	\]
	where 
	\[
		A_{i,f}^{q^m}(C)
		:= 
		\sum_{\bm{u} \in C{[q^m]} , \wt(\bm{u}) = i} 
		\widetilde{f}(\Supp(\bm{u}))\,.
	\]
\end{df}

{Let $C$ be an $[n,k]$ code over~$\FF_{q}$. 
	Then by Remark~\ref{Rem:BachocLem} one can easily 
	observe that, for any positive integer~$m$, 
	the harmonic $q^m$-extended weight enumerator
	$W_{C,f}(x,y;q^m)$ of~$C$ 
	associated to $f\in \Harm_{d}(n)$ 
	can be written as}
\[
{W_{C,f}(x,y;q^m) 
	= 
	(xy)^{d} Z_{C,f}(x,y;q^m),}
\]
{where $Z_{C,f}(x,y;q^m)$ is 
the following homogeneous polynomial of degree $n-2d$:}
\[
{Z_{C,f}(x,y;q^m) 
	= 
	\sum_{i=d}^{n-d}
	A_{i,f}^{q^m}(C)
	x^{n-i-d}y^{i-d}.}
\]

{Now we have the following harmonic analogue of 
the MacWilliams identity 
for $\FF_{q^{m}}$-extended
code. Bachoc~\cite{Bachoc} and Tanabe~\cite{Tanabe2001}
presented the harmonic generalizations of 
the MacWilliams identity 
for binary and non-binary codes, respectively.
Later, an alternative approach of the proof of this 
identity was given in~\cite{CMO20xx}.}

%The proof of this identity is similar to
%the proofs of~\cite[Theorem 2.1]{Bachoc} and~\cite[Theorem 2.1]{Tanabe2001}. 
%One can also follow an alternative way 
%given in~\cite{CMO20xx} to prove this identity. 
%We therefore omit the proof of the following theorem. 

\begin{thm}[MacWilliams type identity]\label{Thm:HarmTupleMacWilliams.} 
	{Let $C$ be an $[n,k]$ code over $\FF_{q}$,
		and let $f \in \Harm_{d}(n)$.
		Then }
%	Let $W_{C,f}(x,y;q^m)$ be 
%	the harmonic $q^m$-extended weight enumerator of 
%	an $[n,k]$ code $C$ over $\FF_{q}$    
%	associated to $f \in \Harm_{d}(n)$. 
%   Then 
%	\[
%		W_{C,f}(x,y;q^m) 
%		= 
%		(xy)^{d} Z_{C,f}(x,y;q^m),
%	\]
%	where $Z_{C,f}(x,y;q^m)$ is a homogeneous polynomial of degree $n-2d$ that satisfies
	\[
		Z_{C^{\perp},f}
		(x,y;q^m)
		= 
		(-1)^{d}
		q^{-(k-d)m} 
%		\frac{(q^{m})^{d}}{|C|} 
		Z_{C,f}
		\left( 
		{x+(q^{m}-1)y}, 
		{x-y};
		q^m
		\right).
	\]
\end{thm}

{Following the approach given in~\cite{CMO20xx},
	we provide the proof of Theorem~\ref{Thm:HarmTupleMacWilliams.}
	in Section~\ref{Sec:Greene}.
	In order to do that
	we need to present some necessary propositions and theorems
	related to~$Z_{C,f}(x,y;q^m)$.}

\begin{df}
	Let $C$ be an $\FF_{q}$-linear code of length~$n$ and
	let $f \in \Harm_{d}(n)$. 
	For any subset $X \subseteq E$ and positive integer~$m$,
	we define
	\begin{align*}
		B_{X}^{q^m}(C)
		& := 
		(q^m)^{\ell(X)}-1,\\
		B_{t,f}^{q^m}(C)
		& := 
		\sum_{X \in E_{t}} 
		\widetilde{f}(X)B_{X}^{q^m}(C).
	\end{align*}
\end{df}

\begin{lem}\label{Lem:ReInterBtf}
	For all $d \leq t \leq n-d$, %we have
    \[
      \displaystyle\sum_{X \in E_{t}} (q^m)^{\ell(X)} = B_{t,f}^{q^m}(C)\,.
    \]
\end{lem}

\begin{proof}
	Since $\sum_{X \in E_{t}} \widetilde{f}(X) = 0$ for  
	$d \leq t \leq n-d$, we can write
	\begin{align*}
		\sum\limits_{X \in E_{t}}
		\widetilde{f}(X)
		(q^m)^{\ell(X)}
%		& =
%		\sum\limits_{X \in E_{t}}
%		\widetilde{f}(X)
%		(((q^m)^{\ell(X)}-1)+1)\\
		& = 
		{\sum\limits_{X \in E_{t}}
		\widetilde{f}(X)
		(B_{X}^{q^m}(C)+1)}\\
		& = 
		B_{t,f}^{q^m}(C)+ \sum\limits_{X \in E_{t}} \widetilde{f}(X)\\
%		& = 
%		B_{t,f}^{q^m}(C)+ 0\\
		& = 
		B_{t,f}^{q^m}(C)\,.\qedhere
	\end{align*}
\end{proof}

\begin{prop}\label{Prop:HarmTupleConnection}
	The relation between $B_{t,f}^{q^m}(C)$ and $A_{i,f}^{q^m}(C)$ 
	is as follows:
	\[
	B_{t,f}^{q^m}(C)
		= 
		(-1)^{d}
		\sum_{i=d}^{n-t} 
		\binom{n-d-i}{t-d} A_{i,f}^{q^m}(C),	
	\]
	if $d \leq t \leq n-d$;
	otherwise, $B_{t,f}^{q^m}(C) = 0$.
\end{prop}

The proof of the above proposition is similar to the proof of~\cite[Proposition 3.8]{CMO20xx}
and is therefore omitted. 
{Using the above proposition, 
%we get the following relation
we have the extended analogue of~\cite[Proposition 3.9]{CMO20xx} as follows.
Since the proof is straightforward, we leave it to the reader.}

\begin{thm}\label{Thm:ExtendedReinter}
	Let $C$ be an $\FF_{q}$-linear code of length~$n$, and let $f \in \Harm_{d}(n)$.
	Then
	\[
	Z_{C,f}(x,y;q^m)
	= 
	(-1)^{d}
	\sum_{t=d}^{n-d} 
	B_{t,f}^{q^m}(C)
	(x-y)^{t-d} y^{n-t-d}.
	\]
\end{thm}

%\begin{proof}
%	By using Proposition~\ref{Prop:HarmTupleConnection},
%	we can prove the theorem with an argument similar to that used to
%	prove~\cite[Proposition 3.9]{CMO20xx}.
%\end{proof}

Using Lemma~\ref{Lem:ReInterBtf}, we have the following proposition.

\begin{prop}\label{Prop:ExtendedRephrase}
	We can {express} $Z_{C,f}(x,y;q^m)$ as follows:
	\[
		Z_{C,f}(x,y;q^m)
		= 
		(-1)^{d}
		\sum\limits_{t=d}^{n-d}\,
		\sum\limits_{X \in E_{t}}
		\widetilde{f}(X)
		(q^m)^{\ell(X)}
		(x-y)^{t-d} y^{n-t-d}.
	\]
\end{prop}

\subsection{The link between the two weight enumerators}

%First we recall the following proposition and lemma from~\cite{JP2013} to 
In this section, we establish two-way relationships between 
the harmonic version of two weight polynomials, 
namely extended weight enumerators and higher weight enumerators.
Each of these relationships plays an important role in the development of the 
subsequent sections.

Let $\Mat(m\times n, \FF_{q})$ be the $\FF_{q}$-linear space 
of $m \times n$ matrices with entries in~$\FF_{q}$.
%Assume that $M \in \Mat(m\times n, \FF_{q})$.
%Then we denote by~$\mathrm{rw}(M,i)$ (resp.\ $\mathrm{col}(M,j)$) 
%the $i$-th row (resp. $j$-th column) of~$M$.
Then for an~$\FF_{q}$-linear code~$C$ of length~$n$,
$S_{m,n}(C)$ denotes the linear subspace of 
$\Mat(m\times n, \FF_{q})$ such that rows of each matrix are in~$C$. 
Then by~\cite[Proposition 5.27]{JP2013}, 
there is an isomorphism
%of $\FF_{q}$-vector spaces 
from $C{[q^m]}$ into~$S_{m,n}(C)$.
For a detailed discussion, we refer the reader to~\cite{JP2013, Simonis1993}.

\begin{thm}\label{Thm:ExtendedtoHigher}
	Let $C$ be an $[n,k]$ code over~$\FF_{q}$ and 
    let $f \in \Harm_{d}(n)$. Then for any positive integer~$m$, we have
%	the harmonic generalization of the extended weight enumerators
%	and the higher weight enumerators of $C$ are related via
	\[
		A_{i,f}^{q^m}(C)
		=
		\sum_{r=0}^{k}
		[m,r]_{q}
		A_{i,f}^{(r)}(C)\,.
	\]
	More specifically, we can write
	\[
		Z_{C,f}
		(x,y;q^m)
		=
		\sum_{r=0}^{k}
		[m,r]_{q}
		Z_{C,f}^{(r)}
		(x,y)\,.
	\]
\end{thm}

\begin{proof}
	Let $\bm{u} \in C{[q^m]}$ and $M \in S_{m,n}(C)$ be 
	the corresponding $m\times n$ matrix under a given 
	isomorphism $C[q^m] \to S_{m,n}(C)$. 
	Let $D_{M} \subseteq C$ be the subcode generated by	the rows of~$M$. 
	Then $\supp(\bm{u}) = \Supp(D_{M})$, and hence $\wt(\bm{u}) = \wt(D_{M})$ 
    (see~\cite[{Lemma 5.4}]{JP2013}). 
    Therefore,
	\begin{align*}
		A_{i,f}^{q^m}(C)
		& =
		\sum_{\bm{u} \in C{[q^m]} , \wt(\bm{u}) = i} 
		\widetilde{f}(\supp(\bm{u}))\\
		& =
		\sum_{M \in S_{m,n}(C), \wt(D_{M}) = i}
		\widetilde{f}(\Supp(D_{M}))\\
		& =
		\sum_{r = 0}^{k}\,
		\sum_{\substack{M \in S_{m,n}(C),\\ \mathrm{rk}({M}) = r}}\,
		\sum_{\substack{D_{M} \subseteq C,\\ \wt(D_{M}) = i}}
		\widetilde{f}(\Supp(D_{M}))\\
		& =
		\sum_{r = 0}^{k}\,
		\sum_{\substack{M \in S_{m,n}(C),\\ \mathrm{rk}({M}) = r}}\,
		\sum_{\substack{D_{M} \in \mathcal{D}_{r}(C),\\ \wt(D_{M}) = i}}
		\widetilde{f}(\Supp(D_{M}))\\
		& =
		\sum_{r=0}^{k}
		[m,r]_{q}
		A_{i,f}^{(r)}(C)
	\end{align*}
	This completes the proof.
\end{proof}

We can also express the harmonic higher weight enumerators
in terms of the harmonic extended weight enumerators.
The following result gives such correspondence.

\begin{thm}\label{Thm:HigherToExtended}
	Let $C$ be an $[n,k]$ code over~$\FF_{q}$ and
    let $f \in \Harm_{d}(n)$. 
    Then for any integer~$r$ such that $0 \leq r \leq k$, 
    we have the following relation:
	\[
		Z_{C,f}^{(r)}
		(x,y)
		=
		\dfrac{1}{[r]_{q}}
		\sum_{j=0}^{r}
		{\bbinom{r}{j}}_{q}
		(-1)^{r-j}
		q^{\binom{r-j}{2}}
		Z_{C,f}
		(x,y;q^j).
	\]
\end{thm}

\begin{proof}
	By Theorem~\ref{Thm:reinter}, Lemmas~\ref{Rem:TBritz} and \ref{Rem:BXrC}  
    and Proposition~\ref{Prop:ExtendedRephrase}, 
	\begin{align*}
		Z_{C,f}^{(r)}(x,y)
		& = 
		(-1)^{d}
		\sum_{t=d}^{n-d} 
		B_{t,f}^{(r)}(C) 
		(x-y)^{t-d} y^{n-t-d}\\
		& = 
		(-1)^{d}
		\sum_{t=d}^{n-d}\,
		\sum_{X \in E_{t}} 
		\widetilde{f}(X)B_{X}^{(r)}(C)
		(x-y)^{t-d} y^{n-t-d}\\
		& = 
		(-1)^{d}
		\sum_{t=d}^{n-d}\,
		\sum_{X \in E_{t}} 
		\widetilde{f}(X)
		{\bbinom{\ell(X)}{r}}_{q}
		(x-y)^{t-d} y^{n-t-d}\\
		& = 
		(-1)^{d}
		\sum_{t=d}^{n-d}\,
		\sum_{X \in E_{t}} 
		\widetilde{f}(X)
		\frac{[\ell(X),r]_{q}}{[r]_{q}}
		(x-y)^{t-d} y^{n-t-d}\\
		& =
		(-1)^{d}
		\frac{1}{[r]_{q}}
		\sum_{t=d}^{n-d}\,
		\sum_{X \in E_{t}}\, 
		\sum_{j=0}^{r} 
		{\bbinom{r}{j}}_{q}
		(-1)^{r-j}
		q^{\binom{r-j}{2}}
		(q^{\ell(X)})^j
		\widetilde{f}(X)\\
		& \hspace{150pt}
		(x-y)^{t-d} y^{n-t-d}\\
		& =
		(-1)^{d}
		\frac{1}{[r]_{q}}
		\sum_{j=0}^{r} 
		{\bbinom{r}{j}}_{q}
		(-1)^{r-j}
		q^{\binom{r-j}{2}}	
		\sum_{t=d}^{n-d}\,
		\sum_{X \in E_{t}}	
		(q^{j})^{\ell(X)}
		\widetilde{f}(X)\\
		& \hspace{150pt}
		(x-y)^{t-d} y^{n-t-d}\\
		& =
		\frac{1}{[r]_{q}}
		\sum_{j=0}^{r} 
		{\bbinom{r}{j}}_{q}
		(-1)^{r-j}
		q^{\binom{r-j}{2}}
		Z_{C,f}(x,y;q^j).		
	\end{align*}
\end{proof}

\section{A Greene-type theorem}\label{Sec:Greene}

In this section,
we present a Greene-type theorem that relates 
harmonic Tutte polynomials and harmonic higher weight enumerators 
(resp.\ harmonic extended weight enumerators).
As an application of this relation, 
we give a proof of the MacWilliams-type identities given in Theorem~\ref{Thm:HarmHigherMac} 
(resp. Theorem~\ref{Thm:HarmTupleMacWilliams.}).
For the definitions and basic properties
of Tutte polynomials of a matroid 
associated with the (discrete) harmonic function,
we refer the reader to~\cite{CMO20xx}. 
 
\begin{df}
	Let $M = (E,\I)$ be a matroid with rank function $\rho$ and 
	let $f \in \Harm_{d}(n)$. 
	Then the \emph{harmonic Tutte polynomial} of $M$ 
	associated to~$f$ is defined
	as follows: 
	\[
		T(M,f;x,y)
		:=
		{\sum_{X \subseteq E}}
		\widetilde{f}(X)
		(x-1)^{\rho(E)-\rho(X)}
		(y-1)^{|X|-\rho(X)}.
	\] 
\end{df}

{By Remark~\ref{Rem:BachocLem}, 
the following relation between 
the harmonic Tutte polynomial of a matroid and that of its dual is immediately clear:}
\[
  T(M^{\ast},f;x,y) = (-1)^{d} T(M,f;y,x)\,.
\]
Now we have the following Greene-type identity for harmonic extended weight enumerators.

\begin{thm}\label{Thm:HarmExGreene}
	Let $C$ be an $[n,k]$ code over $\FF_{q}$ and 
    let $f \in \Harm_{d}(n)$.
	Then% we have
	\[
		Z_{C,f}(x,y;q^m)
		= 
		(-1)^{d}
		(x-y)^{k-d} y^{n-k-d} 
		T
		\left(
		M_{C},f; 
		\dfrac{x+(q^{m}-1)y}{x-y},
		\dfrac{x}{y}
		\right).
	\]
\end{thm}

\begin{proof}
	The proof follows very similar argument to those in~\cite[Theorem 4.2]{CMO20xx}. 
	We therefore leave the details to the reader.
\end{proof}

\begin{proof}[Proof of Theorem~\ref{Thm:HarmTupleMacWilliams.}]
	{Let 
	$C$ be an $[n,k]$ code,
	and $M_{C}$ be its matroid.
	Let $k^{\perp}$ be the dimension of $C^{\perp}$.
	Then by Theorem~\ref{Thm:HarmExGreene}, 
	\begin{align*}
		(-1)^{d}
		&
		q^{-(k-d)m} 
%		\frac{(q^{m})^{d}}{|C|} 
		Z_{C,f}
		({x+(q^{m}-1)y}, 
		{x-y};
		q^m)\\
		& =
		\frac{(-1)^{2d}}{q^{(k-d)m}}
		(q^m y)^{k-d}
		(x-y)^{n-k-d}
		T\left(M_{C},f;\dfrac{x}{y},\dfrac{x+(q^m-1)y}{x-y}\right)\\
		& =
		(x-y)^{n-k-d}
		y^{k-d}
		T
		\left(
		M_{C},f;
		\dfrac{x}{y},
		\dfrac{x+(q^m-1)y}{x-y}
		\right)\\
		& =
		(-1)^{d}
		(x-y)^{(n-k)-d}
		y^{n-(n-k)-d}
		T
		\left(
		M_{C^{\perp}},f;
		\dfrac{x+(q^m-1)y}{x-y},
		\dfrac{x}{y}
		\right)\\
		& =
		(-1)^{d}
		(x-y)^{k^{\perp}-d}
		y^{n-k^{\perp}-d}
		T
		\left(
		M_{C^{\perp}},f;
		\dfrac{x+(q^m-1)y}{x-y},
		\dfrac{x}{y}
		\right)\\
		& =
		Z_{C^{\perp},f}(x,y)\,.\qedhere
	\end{align*}}
%	Hence Theorem is proved. 
\end{proof}

\begin{proof}[Proof of Theorem~\ref{Thm:HarmHigherMac}]
%	Lemma~\ref{Lem:Bachoc} and Remark~\ref{Rem:BachocLem}
%	implies that~$Z_{C,f}^{(r)}(x,y)$ is a polynomial.
	By Theorems~\ref{Thm:HigherToExtended} and~\ref{Thm:HarmTupleMacWilliams.}, 
	we can write
	\begin{align*}
		Z_{C^{\perp},f}^{(r)}
		(x,y)
		& =
		\dfrac{1}{[r]_{q}}
		\sum_{j=0}^{r}
		{\bbinom{r}{j}}_{q}
		(-1)^{r-j}
		q^{\binom{r-j}{2}}
		Z_{C^{\perp},f}
		(x,y;q^j)\\
		& =
		(-1)^{d} 
		\dfrac{1}{[r]_{q}}
		\sum_{j=0}^{r}
		{\bbinom{r}{j}}_{q}
		(-1)^{r-j}
		q^{\binom{r-j}{2}}
		\frac{(q^{j})^{d}}{q^{jk}}\\
		& \hspace{100pt}
		Z_{C,f}
		\left( 
		{x+(q^{j}-1)y}, 
		{x-y};
		q^j
		\right)\\
		& =
		\dfrac{1}{[r]_{q}}
		\sum_{j=0}^{r}
		{\bbinom{r}{j}}_{q}
		(-1)^{r+d-j}
		q^{\binom{r-j}{2}-j(k-d)}\\
		& \hspace{100pt}
		Z_{C,f}
		\left( 
		{x+(q^{j}-1)y}, 
		{x-y};
		q^j
		\right)\\
		& =
		\dfrac{1}{[r]_{q}}
		\sum_{j=0}^{r}
		{\bbinom{r}{j}}_{q}
		(-1)^{r+d-j}
		q^{\binom{r-j}{2}-j(k-d)}
		\sum_{\ell=0}^{j}
		[j,\ell]_{q}\\
		& \hspace{100pt}
		Z_{C,f}^{(\ell)}
		\left( 
		{x+(q^{j}-1)y}, 
		{x-y}
		\right)\\
		& =
		\dfrac{1}{[r]_{q}}
		\sum_{j=0}^{r}\,
		\sum_{\ell=0}^{j}
		{\bbinom{r}{j}}_{q}
		(-1)^{r+d-j}
		q^{\binom{r-j}{2}-j(k-d)}
		[j,\ell]_{q}\\
		& \hspace{100pt}
		Z_{C,f}^{(\ell)}
		\left( 
		{x+(q^{j}-1)y}, 
		{x-y}
		\right).
	\end{align*}
	The proof is completed by the following easily-proven identity:
	\[
	\dfrac{1}{[r]_{q}}
	{\bbinom{r}{j}}_{q}
	[j,\ell]_{q}
	=
	\frac{1}{q^{j(r-j)} [r-j]_{q} q^{\ell(j-\ell)} [j-\ell]_{q}}\,.\qedhere
	\]
	%It is not hard to prove this identity, so we omit it.	
\end{proof}

\begin{ex}
	Let $C$ be a $[5,2]$ code over $\FF_{2}$ with generator matrix:
	\[
	\begin{pmatrix}
		1 & 1 & 1 & 0 & 0\\
		0 & 0 & 0 & 1 & 1
	\end{pmatrix}.
	\]
	Then the dual code~$C^{\perp}$ has the generator matrix:
	\[
	\begin{pmatrix}
		1 & 0 & 1 & 0 & 0\\
		0 & 1 & 1 & 0 & 0\\
		0 & 0 & 0 & 1 & 1
	\end{pmatrix}.
	\]
	{We consider $f \in \Harm_{1}(5)$, where
	$f = a\{1\}+b\{2\}+c\{3\}+d\{4\}-(a+b+c+d)\{5\}$.}
	By direct calculation, we have
	\begin{align*}
		W_{C,f}^{(0)} & = 0,\\
		W_{C,f}^{(1)} & = (a+b+c) (xy)^{1} (xy^2-x^2y),\\
		W_{C,f}^{(2)} & = 0.
	\end{align*}
	Here $Z_{C,f}^{(0)} = Z_{C,f}^{(2)} = 0$ and 
	$Z_{C,f}^{(1)} = (a+b+c)(xy^2 - x^2y)$.
	It is evident that $C^{\perp}$ has seven subcodes of dimension~$2$.
	The explicit list of subcodes is given in~\cite{BrCa22}.
	Now by Theorem~\ref{Thm:HarmHigherMac}, we have
	\begin{align*}
		Z_{C^{\perp},f}^{(2)}(x,y)
		& =
		\sum_{j = 0}^{2}
		\sum_{\ell = 0}^{j}
		(-1)^{3-j}
		\frac{2^{{2-j \choose 2}-j(2-j)-\ell(j-\ell)-j}}{[2-j]_{2} [j-\ell]_{2}}
		Z_{C,f}^{(\ell)}
		\left( 
		{x+(2^{j}-1)y}, 
		{x-y}
		\right)\\
		& =
		(a+b+c)(xy^2-y^3).
	\end{align*}
	Hence $W_{C^{\perp},f}^{(2)}(x,y) = (a+b+c)(xy)(xy^2-y^3)$.
\end{ex}

{The following remarks concerning harmonic generalization of
the MacWilliams identity for $r$-th higher weight enumerators 
stated in Theorem~\ref{Thm:HarmHigherMac} might be helpful:}

\begin{rem}\label{Rem:Krawtch}
	{Using the binomial series expansion for $i = d,\ldots, n-d$,
    we obtain the identity
	\begin{align*}
		(x+(q^j-1)y)^{n-d-i}
		(x-y)^{i-d}
		& =
		\sum_{p=d}^{n-d}
		K_{p}(i-d;n-2d,q^{j})
		x^{n-d-p}y^{p-d},
	\end{align*}
	where $K_{p}(i-d; n-2d, q^j)$ is known as the Krawtchouk polynomial. 
	For any positive integer~$n > 2d$, 
	it is defined as follows (see~\cite{Simonis1993}):
	\[
		K_{p}(i-d; n-2d, q^j)
		:=
		\sum_{m=0}^{p}
		(-1)^{m}
		(q^{j}-1)^{p-m}
		\binom{i-d}{m}
		\binom{n-d-i}{p-m}.
	\]}
\end{rem}

\begin{rem}\label{Rem:EquivMac}
	{Since $A_{i,f}^{(0)}(C^{\perp}) = 0$, 
		we can write $A_{i,f}^{(r)}(C^{\perp})$ explicitly 
		for $i = d, \ldots, n-d$ and $r\neq 0$ as
	\begin{align*}
%		\sum_{i = d}^{n-d}
		A_{i,f}^{(r)}(C^{\perp})
%		x^{n-i}y^{i}
%		Z_{C^{\perp},f}^{(r)}
%		(x,y)
		= 
		(-1)^{d} 
		\sum_{\ell=1}^{r}\,
		&\sum_{j=\ell}^{r}
		(-1)^{r-j}
		\frac{q^{{r-j \choose 2}-j(r-j)-\ell(j-\ell)-j(k-d)}}{[r-j]_{q} [j-\ell]_{q}}\\
		&
		\sum_{p = d}^{n-d}
		K_{p}(i-d;n-2d,q^{j})
		A_{p,f}^{(\ell)}(C).
	\end{align*}}
\end{rem}

\begin{thm}\label{Thm:HarmHigherGreene}
	Let $C$ be an $[n,k]$ code over $\FF_{q}$ and
    let $f \in \Harm_{d}(n)$.
	Then we have the following Greene-type relation:
	\begin{align*}
		Z_{C,f}^{(r)}(x,y)
		= 
		\frac{1}{[r]_{q}}
		\sum_{j=0}^{r}
		{\bbinom{r}{j}}_{q}
		(-1)^{r+d-j}
		q^{\binom{r-j}{2}}
		&(x-y)^{k-d} y^{n-k-d}\\
		&T
		\left(
			M_{C},f; 
			\dfrac{x+(q^{j}-1)y}{x-y},
			\dfrac{x}{y}
		\right).
	\end{align*}
\end{thm}

\begin{proof}
	Apply Theorems~\ref{Thm:HigherToExtended} and~\ref{Thm:HarmExGreene}.
\end{proof}

\section{Applications to subcode designs}\label{Sec:Design}

Britz and Shiromoto~\cite{BrSh2008} gave a %celebrated 
generalization of the Assmus-Mattson Theorem~\cite{AsMa69}
by introducing designs from subcode supports of an $\FF_{q}$-linear code;
see Theorem~\ref{Thm:AssmusMattson} below. 
In this section, we provide a new proof of this generalization
%e Assmus-Mattson Theorem for subcode designs
by using the harmonic higher weight enumerators of an $\FF_{q}$-linear code.  
Our method of proof is inspired by Bachoc~\cite{Bachoc} 
who first used the concept
of harmonic functions to prove the Assmus-Mattson Theorem for $t$-designs.

A $t$-$(n, k, \lambda)$ design is a collection $\mathcal{B}$ of $k$-subsets (called blocks) 
of a set $E$ of $n$ elements,
such that any $t$-subset of $E$ is contained in exactly $\lambda$ blocks. 
A design is called \emph{simple} if the blocks are distinct; 
otherwise, the design is said to have \emph{repeated blocks}.

The following theorem gives a remarkable characterization
of $t$-designs in terms of the harmonic spaces.

\begin{thm}[\cite{Delsarte}]\label{Thm:Delsert}
	Let $i,t$ be integers such that $0 \leq t \leq i \leq n$.
	A subset $\mathcal{B} \subseteq E_{i}$ is a $t$-design 
	if and only if
	$\sum_{B\in \mathcal{B}}(\sum_{A \in E_{d}, A \subseteq B} f(A)) = 0$
	for all $f \in \Harm_{d}(n)$, $1 \leq d \leq t$.
\end{thm}

\begin{ex}
	{Let $C$ be a $[6,3]$ code over $\FF_{2}$ with generator matrix
	\[
	\begin{pmatrix}
		1 & 1 & 0 & 0 & 0 & 0\\
		0 & 0 & 1 & 1 & 0 & 0\\
		0 & 0 & 0 & 0 & 1 & 1
	\end{pmatrix}.
	\]
	There are seven linear subcodes of~$C$ with dimension~$2$. 
	Three of them have weight~$4$; 
    the set of supports of these three is
	\[
		\mathcal{S}_{2,4}
		=
		\{\{1,2,3,4\},\{1,2,5,6\},\{3,4,5,6\}\}\,.
	\] 
	Note that $\mathcal{S}_{2,4}(C)$ is a $1$-design. 
	Moreover, for $f \in \Harm_{1}(6)$, where
	$f = a\{1\}+b\{2\}+c\{3\}+d\{4\}+e\{5\}-(a+b+c+d+e)\{6\}$,
	we get $A_{4,f}^{(2)}(C) = 0$.}
\end{ex}

Now we have the following useful lemma.

\begin{lem}\label{Lem:t-design}
	Let $C$ be an $[n,k]$ code over~$\FF_{q}$.
	Let $r,i$ be the positive integers such that 
	$r \leq k$ and $i \geq d_{r}$.
	Then the set $\mathcal{S}_{r,i}(C)$ 
	forms a $t$-design 
	if and only if
	{$A_{i,f}^{(r)}(C) = 0$} for all
	$f \in \Harm_{d}(n)$, $1 \leq d \leq t$.
\end{lem}

\begin{proof}
	By the definition of~$\widetilde{f}$ from~(\ref{Equ:TildeF}), we have
	\begin{align*}
		A_{i,f}^{(r)}(C) 
		& =
		\sum_{\substack{D \in \mathcal{D}_{r}(C),\\ \wt(D) = i}}
		\widetilde{f}(\Supp(D))\\
		& =
		\sum_{\substack{D \in \mathcal{D}_{r}(C),\\ \wt(D) = i}}
		\sum_{\substack{Z\in E_{d},\\ Z\subseteq \Supp(D)}}
		f(Z)\\
		& =
		\sum_{Y \in \{\Supp(D) \mid D \in \mathcal{D}_{r}(C), \wt(D) = i\}}
		\sum_{\substack{Z \in E_{d}, Z \subseteq Y}}
		f(Z)\\
		& =
		\sum_{Y \in \mathcal{S}_{r,i}(C)}
		\sum_{\substack{Z \in E_{d}, Z \subseteq Y}}
		f(Z).
	\end{align*}
	Hence, the proof follows by Theorem~\ref{Thm:Delsert}.
\end{proof}

Below is the subcode generalization of the Assmus-Mattson Theorem.
\begin{thm}[\cite{BrSh2008}]\label{Thm:AssmusMattson}
	Let $C$ be an $[n,k,d_{min}]$ code over~$\FF_{q}$.
	Let $r,t \geq 1$ be integers with~$r \leq k$
	and $t \leq d_{min}$.
	Define
	\[
		\mathcal{L}_{r,t}
		:=
		\bigg\{
		i \in \{d_{r}^{\perp},\ldots,n-t\}
		\;\bigg|\;
        {\sum_{\ell=0}^{r}
		[r]_{\ell}\, A_{i}^{(\ell)}(C^{\perp})} \neq 0
		\bigg\},
	\]
	where $d_{r}^{\perp}$ is the $r$-th generalized Hamming weight
	of the dual code~$C^{\perp}$.
	Suppose that $|\mathcal{L}_{\mu,t}| \leq d_{\mu}-t$
	for all $\mu = 1,\ldots, r$.
	Then for all $i \geq d_{r}$ and 
	$j = d_{r}^{\perp}, \ldots, n-t$,
	$\mathcal{S}_{r,i}(C)$ and $\mathcal{S}_{r,j}(C^{\perp})$
	each form a (not necessarily simple) $t$-design.
\end{thm}

\begin{proof}
	Let $f \in \Harm_{d}(n)$ for $1\leq d \leq t$.
	By Lemma~\ref{Lem:t-design}, 
    it is sufficient to show that $A_{i,f}^{(r)}(C) = 0$
	(resp. $A_{i,f}^{(r)}(C^{\perp}) = 0$)
	for all $i = d,\ldots,n-d$.
	Applying Theorem~\ref{Thm:HarmHigherMac}
	and {Remark~\ref{Rem:EquivMac}} to the dual code $C^\perp$
    gives the identity    
%	we can explicitly write
%	$A_{i,f}^{(r)}(C)$
%	as
	\begin{equation}\label{Equ:GenKrch}
		(-1)^{d}
		\sum_{\ell = 1}^{r}
		M_{\ell}^{r}
		\underline{A}^{(\ell)}(C^{\perp};f)
		=
		\underline{A}^{(r)}(C;f),
	\end{equation}
%    {\color{red}In the expression above, a matrix is multiplied by $A_{i,f}^{(\ell)}(C^{\perp})$ 
%    but it is not clear that $A_{i,f}^{(\ell)}(C^{\perp})$ is a vector indexed by $i$; 
%    nor is it clear how $A_{p+d,f}^{(r)}(C)$ can be interpreted as a vector of size $n-2d+1$... 
%    could this be written more clearly (or corrected?) ?}
	where 
%	$\underline{A}^{(\ell)}(C^{\perp};f)$ 
%	(resp. $\underline{A}^{(r)}(C;f)$) is a vector of size
%	$n-2d+1$ defined as
	$\underline{A}^{(\ell)}(C^{\perp};f) := (A_{d,f}^{(\ell)}(C^{\perp}), \ldots, A_{n-d,f}^{(\ell)}(C^{\perp}))$
	(resp. $\underline{A}^{(r)}(C;f) := (A_{d,f}^{(r)}(C), \ldots, A_{n-d,f}^{(r)}(C))$)
    and where $M_{\ell}^{r}$ is the square matrix of order~$n-2d+1$
    defined for all $i=d,\ldots,n-d$ and $p=0,\ldots,n-2d$ by
	\begin{equation*}
		(M_{\ell}^{r})_{p,i-d}
		:=				
%		\sum_{j = 0}^{p}
%		(-1)^{p+j}
%		\binom{n-2d-j}{n-2d-p}
%		q^{\ell(j-k+\ell-r)}
%		{\bbinom{j-k}{r-\ell}}_{q}
%		\binom{n-d-i}{j},
		\sum_{j=\ell}^{r}
		(-1)^{r-j}
		\frac{q^{{r-j \choose 2}-j(r-j)-\ell(j-\ell)-j(k^{\perp}-d)}}{[r-j]_{q} [j-\ell]_{q}}
		K_{p}(i-d;n-2d,q^j),
	\end{equation*}
	where $k^{\perp}$ denotes the dimension of~$C^{\perp}$ and 
	{$K_{p}(i-d;n-2d,q^j)$ is the Krawtchouk polynomial 
	given in Remark~\ref{Rem:Krawtch}.}
%	by (see~\cite{Simonis1993})
%	\[
%		K_{p}(i-d; n-2d, q^j)
%		:=
%		\sum_{m=0}^{p}
%		(-1)^{m}
%		(q^{j}-1)^{p-m}
%		\binom{i-d}{m}
%		\binom{n-d-i}{p-m}.
%	\]
	Since~$d_{r}$ is the~$r$-th generalized Hamming weight of~$C$, 
	\begin{equation}\label{Equ:Independent}
		A_{d,f}^{(r)}(C)
		=
		A_{d+1,f}^{(r)}(C)
		=
		\cdots
		=
		A_{d_{r}-1,f}^{(r)}(C)
		=
		0\,.
	\end{equation}
	For $r = 1$, this leads to $d_{1} -d$
	independent equations in the terms
	$A_{i,f}^{(1)}(C^{\perp})$
	for 
	$d \leq i \leq n - d$.
%	By hypothesis 
	Since $|\mathcal{L}_{1,d}| \leq d_{1}-d$
	and $1 \leq d \leq t \leq d_{\min}$,
%	{\color{red}Which hypothesis?}, 
	there are at most 
	$d_{1}-d$
	unknowns and so the only solution is trivial.
	Hence, $A_{i,f}^{(1)}(C^{\perp}) = 0$
	for all $i = d, \ldots, n -d$.
	
	Similarly, the $n - 2d + 1$ equations 
	$A_{i,f}^{(1)}(C^{\perp}) = 0$, 
	$d \leq i \leq n - d$, 
	each translate into an equation in the terms  
	$A_{i,f}^{(1)}(C)$,
	$d \leq i \leq n - d$, 
	using equation~(\ref{Equ:GenKrch}) applied to the dual code~$C^{\perp}$.
%	{\color{red}dual code?}. 
	Since these equations are independent, the only solution is trivial;
	that is,
	$A_{i,f}^{(1)}(C) = 0$
	for all $i=d,\ldots,n - d$. 
	
	Now let $r = m+1$ for any positive integer $m$ such that $1 \leq m \leq k-1$. 
	Then by induction, we have from~(\ref{Equ:Independent}) that 
	there are $d_{m+1}-d$ independent equations in the terms $A_{i,f}^{(m+1)}(C^{\perp})$
	for $d \leq i \leq n-d$ as 
	\[
		A_{d,f}^{(\ell)}(C^{\perp})
		=
		A_{d+1,f}^{(\ell)}(C^{\perp})
		=
		\cdots
		=
		A_{n-1,f}^{(\ell)}(C^{\perp})
		=
		0
	\]
	for all $1 \leq \ell \leq m$.
	Hence, by a similar argument to that above, 
	$A_{i,f}^{(m+1)}(C^{\perp}) = 0$
	(resp. $A_{i,f}^{(m+1)}(C) = 0$)
%	for all $i \leq n -t$ and $d \leq t$, and
	for all $i = d,\ldots,n - d$.
	This completes the proof.
\end{proof}

%%%%%%%%%%%%%%%%%%%%%%%%%%%%%%%%%

\section*{Acknowledgements}
This work was supported by JSPS KAKENHI (22K03277) and SUST Research Centre (PS/2023/1/22).
The authors would also like to thank the anonymous
reviewers for their beneficial comments on an earlier version of the manuscript.

\section*{Data availability statement}
The data that support the findings of this study are available from
the corresponding author.

\end{document}